\documentclass[11pt]{article}

\usepackage{amsmath,amssymb,amsfonts,amsthm}
\usepackage{graphicx,float,url,color}
\usepackage{lineno}

\usepackage[left=4cm,right=4cm,top=4cm,bottom=4cm]{geometry}

\usepackage[colorlinks=true]{hyperref}
\usepackage{pgfplots}
\pgfplotsset{compat=newest}
\usetikzlibrary{shapes}
\usetikzlibrary{plotmarks}
\usepackage{tikz}
\usepackage{caption}
\usepackage{dsfont}
\usepackage{algorithmic}
\usepackage{lipsum}
\usepackage{amsfonts}
\usepackage{graphicx}
\usepackage{algorithm}
\graphicspath{{fig/}}

\usepackage[scientific-notation=true]{siunitx}

\usepackage{subfigure}

\usepackage{comment}

\usepackage[normalem]{ ulem }   
\usepackage{soul}   
\def\RR{\mathbb{R}}
\def\NN{\mathbb{N}}
\def\eps{\varepsilon}
\renewcommand{\geq}{\geqslant}
\renewcommand{\leq}{\leqslant}

\def\herd{\mathrm{herd}}
\def\lock{\mathrm{lock}}
\def\equi{\mathrm{equi}}
\def\cC{{\mathcal C}}

\def\cR{{\mathcal R}}
\def\cU{{\mathcal U}}

\definecolor{mygreen}{rgb}{0,0.7,0}
\definecolor{myblue}{rgb}{0,0,1}
\definecolor{myred}{rgb}{1,0,0}

\newtheorem{theorem}{Theorem}

\newtheorem{proposition}{Proposition}

\newtheorem{lemma}{Lemma}

\theoremstyle{definition}\newtheorem{remark}{Remark}

\begin{document}


\title{How Best Can Finite-Time Social Distancing Reduce Epidemic Final Size?}

\author{Pierre-Alexandre Bliman\footnote{Corresponding author} \footnote{Inria, Sorbonne Universit\'e, Universit\'e Paris-Diderot SPC, CNRS, Laboratoire Jacques-Louis Lions, \'equipe MAMBA, Paris, France. 
\texttt{Pierre-Alexandre.bliman@inria.fr}}
and Michel Duprez\footnote{Inria, Universit\'e de Strasbourg,  ICUBE, \'equipe MIMESIS, Strasbourg, France.
\texttt{michel.duprez@inria.fr}}}



\maketitle

\begin{abstract}
Given maximal social distancing duration and intensity, how can one minimize the epidemic final size, or equivalently the total number of individuals infected during the outbreak?
A complete answer to this question is provided and demonstrated here for the SIR epidemic model.
In this simplified setting, the optimal solution consists in enforcing the highest confinement level during the longest allowed period, beginning at a time instant that is the unique solution to certain 1D optimization problem.
Based on this result, we present numerical results showing the best possible performance for a large set of basic reproduction numbers and lockdown durations and intensities.
\end{abstract}






\section{Introduction}

The current outbreak of Covid-19 and the entailed implementation of social distancing on an unprecedented scale, leads to a renewed interest in modelling and analysis of the non-pharmaceutical intervention strategies to control infectious diseases.
In contrast to the remotion of susceptible individuals (by vaccination) or infectious individuals (by isolation or quarantine) from the process of disease transmission, the term ``social distancing" refers to attempts to directly reduce the infecting contacts within the population.
Such actions
may be obtained through voluntary actions, possibly fostered by government information campaigns, or by mandatory measures such as partial or total lockdown.
Notice that, when no vaccine or therapy is available, such containment strategies constitute probably the only mid-term option.

Optimal control approaches have been abundantly explored in the past in the framework of control of transmissible diseases, see e.g.\ \cite{Lenhart:2007aa,Sharomi:2017aa} and bibliographical references in \cite{Bliman:2020aa}.
Optimal control of social distancing (possibly coupled with vaccination, treatment or isolation) is usually considered through the minimization of a finite-time integral cost linear in the state, and quadratic in the input control variables or jointly bilinear in the two signals \cite{Behncke:2000aa,Yan:2008aa,Lee:2010aa,Lin:2010aa,Alvarez:2020aa,Djidjou-Demasse:2020aa}.
The authors of \cite{Morris:2020aa} study the optimal control allowing to minimize the maximal value taken by the infected population.
The integral of the deviation between the natural infection rate and its effective value due to confinement is used as a cost in \cite{Miclo:2020aa}, together with constraints on the maximal number of infected.
Optimal public health interventions as a complement to vaccination campaigns have been studied in \cite{Buonomo:2019aa,Buonomo:2019ab}; see also \cite{Manfredi:2013aa} for more material on behavioral epidemiology.

The magnitude of the outbreak, usually called the epidemic {\em final size}, is another important characteristic.
It is defined as the total number of initially susceptible individuals that become infected during the course of the epidemic.
Abundant literature exists concerning this quantity, since Kermack and Mc Kendrick's paper from 1927 \cite{Kermack:1927aa}; see \cite{Ma:2006aa,Andreasen:2011aa,Katriel:2012aa,Miller:2012aa} for important contributions to its computation in various deterministic settings.
Recently, optimal control approach has been introduced to minimize the final size by temporary reduction of the contact rate on a given time interval $[0,D]$, $D>0$.
This issue has been considered in \cite{Ketcheson:2020aa}, with total lockdown and added integral term accounting for control cost; and in \cite{Bliman:2020aa}, where partial lockdown is considered as well.
The corresponding optimal control is bang-bang, with maximal distancing intensity applied on a subinterval $[T_0^*,D]$, for some unique $T_0^*\in [0,D)$ depending of the initial conditions, and no action otherwise. 

In a population in which a large proportion of individuals is immune, either after vaccination or after having been infected, the infection is more likely to be disrupted.
The {\em herd immunity} threshold is attained when the number of infected individuals begins to decrease over time.
While the epidemic final size is {\em always} smaller than this value, a significant proportion of susceptible individuals may still be infected until the epidemic is over.
In this perspective, minimizing the epidemic final size can be seen as an attempt to stop the outbreak as close as possible after passing the herd immunity threshold.

While distancing enforcement cannot last for a long time, there is indeed  no reason in practice why it should be restricted to start at a given date ---typically ``right now".
Elaborating on \cite{Bliman:2020aa}, we consider in the present paper a more general optimal control problem, achieved through social distancing during a given maximal time duration $D>0$, but without prescribing the onset of this measure.
A key result below (Theorem \ref{th:T}) shows the existence of a unique time $T^*$, which depends upon the initial conditions, for which the optimal control corresponds to applying maximal distancing intensity on the interval $[T^*,T^*+D]$: this more natural setting yields a more efficient control strategy.

The paper is organized as follows.
We introduce in Section \ref{se1} the precise setting of the problem under study and formulate the three main results: Theorem \ref{th:T} demonstrates the existence and uniqueness of the optimal policy and provides a constructive characterization; Theorem \ref{th2} studies its dependence upon the lockdown intensity and duration; Theorem \ref{th3} shows that above a certain critical lockdown intensity, optimal social distancing on a sufficiently long period approaches herd immunity arbitrarily close.
{ Section \ref{se3}
provides extensive numerical essays.}
The proof of Theorem \ref{th:T} is  the subject of Section \ref{se2}.
Concluding remarks are given in Section \ref{se5}.
{ Comments on the numerical implementation of the optimum search method and the related algorithms are provided in \ref{se6}.}

\section{Problem description and main results}
\label{se1}

Consider the system
\begin{equation}\label{SIR}
\begin{array}{ll}
\vspace{.1cm}
\dot S(t) = -u(t)\beta S(t)I(t), & t\geqslant 0\\
\dot I(t) = u(t) \beta S(t)I(t) - \gamma I(t), &t\geqslant 0
\end{array}
\end{equation}
complemented with nonnegative initial data $S(0)=S_0$, $I(0)=I_0$ such that $S_0+I_0 \leq 1$.
The input $u$, taking on values in $[0,1]$, models the effect of a social distancing policy: $u(t)=1$ corresponds to absence of restrictions, while $u(t)=0$, corresponding to complete lockdown, prohibits any contact and thus any transmission.
In the sequel, we call {\em uncontrolled system} the system corresponding to $u\equiv 1$, and generally speaking restrict $u\in L^\infty(0,+\infty)$ to be such that $\alpha \leq u(t) \leq 1$ for a given constant $\alpha\in [0,1)$ and for almost any $t\geq 0$.
The constant $\alpha$, called here the {\em maximal lockdown intensity}\footnote{Therefore, a {\em smaller} value of the maximal lockdown intensity $\alpha$ may produce {\em more intense} lockdown.} determines the most intense achievable social distancing.

We assume in all the sequel that the basic reproduction number $\cR_0$ of the uncontrolled system fulfils:
$$
\cR_0 := \frac{\beta}{\gamma} > 1.
$$
This constant fully characterizes the dynamics of this system.

For any $u$ as above, one defines
\begin{equation*}
S_\infty(u):=\lim\limits_{t\rightarrow\infty}S(t),
\end{equation*}
for $(S,I)$ the solution to \eqref{SIR}.
The quantity $1-S_\infty(u)$ is the proportion of retired individuals after the outbreak.
It is called the {\em attack ratio}, or the {\em epidemic final size} when numbers of individuals are considered instead of proportions.
This notion plays a central role in the sequel.

For the uncontrolled model \eqref{SIR} (with $u\equiv 1$), the herd immunity is
\begin{equation}
\label{eq63}
S_\herd := \frac{\gamma}{\beta} = \frac{1}{\cR_0}.
\end{equation}
Any equilibrium $(S_\equi,0)$, $0 \leq S_\equi \leq 1$, of this system is {\em stable} if $0\leq S_\equi \leq S_\herd$, {\em unstable} if $S_\herd < S_\equi$, so that the disease prospers if introduced in the population with $S_0 > S_\herd$  (before it finally fades away), and dies out otherwise.
Coherently with this observation, if $u(t)$ equals $1$ after a finite time, then one has
$$
S_\infty(u) \leq S_\herd.
$$ 
In this optic, attempting to reduce the epidemic final size by finite-time intervention is equivalent to try to stop as closely as possible from the herd immunity threshold.

For any $0< T\leq T'$ and $\alpha\in[0,1)$, let $\mathcal{U}_{\alpha,T,T'} $ be defined by
\begin{equation*}
\mathcal{U}_{\alpha,T,T'} := \{u \in L^\infty(0,+\infty),\ \alpha\leq u(t) \leq 1 \text{ if } t\in [T,T'], u(t)=1 \text{ otherwise} \}.
\end{equation*}
We also consider the set of those functions $u_{T,T'}$ of $\cU_{\alpha,T,T'}$ defined by
\begin{equation}
\label{eq55}
u_{T,T'}=\mathds{1}_{[0,T]}+\alpha \mathds{1}_{[T,T']}+\mathds{1}_{[T',+\infty)},
\end{equation}
where the notation $\mathds{1}_\cdot$ denotes characteristic functions\footnote{That is e.g.\ $\mathds{1}_{[0,T]}(t)=1$ if $t\in [0,T]$, 0 otherwise.},
and denote $\mathbf 1$ the function of $L^\infty(0,+\infty)$ equal to 1 (almost) everywhere.


The main result of the paper is now given.
It indicates how to optimally implement distancing measures, in order to minimize the epidemic final size.
To state this result, introduce first the function $\psi$ given by
\begin{equation}
\label{eq:psi new}
\psi : T\ni [0,\infty)\mapsto  -\frac{I^{T}(T+D)}{I^{T}(T)}+(\alpha-1)\gamma \int_{T}^{T+D}\frac{I^{T}(T+D)}{I^{T}(t)}dt+1,
\end{equation} 
where $(S^{T}, I^{T})$ denotes the solution to \eqref{SIR} with $u=u_{T,T+D}$ defined in \eqref{eq55}.

\begin{theorem}\label{th:T}
For any $\alpha\in [0,1)$ and $D>0$, the optimal control problem
\begin{equation}
\label{OCP2}\tag{$\mathcal{P}_{\alpha,D}$}
\sup_{T\geq 0}\ \sup_{u\in\cU_{\alpha, T,T+D}} S_\infty(u)
\end{equation}
 admits a unique solution.
 The optimal control is equal to the function $u_{T^*,T^*+D}$ defined in \eqref{eq55}, where the value $T^*\geq 0$ is characterized by the fact that;
 \begin{itemize}
 \item if $\psi(0)\geq 0$, then $T^*=0$;
 \item if $\psi(0)<0$, then $T^*$ is the unique solution to 
 \begin{equation}\label{eq:psi=0}
 \psi(T^*)=0.
 \end{equation}
 \end{itemize}
Moreover, if $T^*>0$, then $S(T^*)>S_\herd$ if $\alpha >0$, and $S(T^*)=S_\herd$ if $\alpha=0$.
Last, fixing $S_0\in(S_\herd,1)$, it holds
\[
\lim\limits_{I_0\searrow 0^+} T^* = +\infty.
\]

\end{theorem}

For subsequent use, we denote $(S^*,I^*)$ the optimal solution, and $S_\infty^*$ the value function of problem \eqref{OCP2}, that is by definition:
\begin{equation}
\label{eq56}
S_\infty^* = S_\infty^*(S_0,I_0) := \sup_{T\geq 0}\ \sup_{u\in\cU_{\alpha, T,T+D}} S_\infty(u).
\end{equation}

Theorem \ref{th:T} establishes that, among all intervention strategies carried out on a time interval of length $D$  with an intensity located at each time instant between $\alpha$ and $1$, a single one minimizes the epidemic final size.
The corresponding control is bang-bang and consists in enforcing the most intense social distancing level $\alpha$ on the time interval $[T^*,T^*+D]$, where $T^*\geq 0$ is uniquely assessed in the statement.
The value of $T^*$ depends upon the initial value $(S_0,I_0)$ through the solution $(S^{T}, I^{T})$ of System \eqref{SIR} appearing in the expression \eqref{eq:psi new}.

Assessing the value of $\psi(T)$ for given $T\geq 0$ amounts to solve the ordinary differential equation \eqref{SIR} and to evaluate the quantity in \eqref{eq:psi new} ---tasks routinely achieved by standard scientific computational environments.
It is shown in the proof of Theorem \ref{th:T} (Section \ref{se23}) that, if $\psi(0)<0$, then $\psi$ is negative on $(0,T^*)$ and positive on $(T^*,\infty)$.
This remark permits implementation of an efficient bisection algorithm to assess the optimal value $T^*$; see more details on the implementation aspects in \ref{se6}.

We continue with some properties characterizing the dependence of the value function with respect to the parameters.
\begin{theorem}\label{th2}
The value function $S_\infty^*$ is increasing with respect to the parameter $D>0$ and decreasing with respect to the parameter $\alpha\in [0,1)$. 
\end{theorem}

The statement of Theorem \ref{th2} corresponds to the intuition whereby longer or more intense interventions result in greater reduction of the epidemic final size. 

\begin{proof}[Proof of Theorem \ref{th2}]
Let $0 < D \leq D'$ and $1 > \alpha \geq \alpha' \geq 0$, with $(D,\alpha)\neq(D',\alpha')$, and denote for short $S_\infty^*$ and $S_\infty^{'*}$ the corresponding optimal costs.
From \eqref{eq56} and the observation that $\cU_{\alpha, T,T+D}\subset\cU_{\alpha', T,T+D'}$, one deduces easily that $S_\infty^*\leq S_\infty^{'*}$.
Assume by contradiction that $S_\infty^*= S_\infty^{'*}$.
Then the optimal value $S_\infty^{'*}$ is realized for two different optimal controls: one in $\cU_{\alpha, T,T+D}$ and one in $\cU_{\alpha', T,T+D'}\setminus\cU_{\alpha, T,T+D}$.
This contradicts the uniqueness of the optimal control, demonstrated in Theorem \ref{th:T}.
One thus concludes that $S_\infty^*< S_\infty^{'*}$.
\end{proof}

Theorem \ref{th2} leads to the following question: what is the benefit of increasing indefinitely the lockdown duration $D$, and is it possible by this mean to stop the disease spread arbitrarily close to the herd immunity?
The next result answers tightly this issue.

\begin{theorem}\label{th3}
For any $S_0\in(S_\herd,1)$, define
  \begin{equation}
  \label{eq323}
   \overline{\alpha}:=\frac{S_\herd}{S_0+I_0-S_\herd}(\ln S_0 - \ln S_\herd).
  \end{equation}
Then $\overline\alpha\in (0,1)$ and 
the following properties are fulfilled.
  \begin{enumerate}
    \item[(i)]
    If $\alpha\in [0,\overline{\alpha}]$, then
\begin{equation}
\label{eq324}
    \lim_{D\to + \infty} S_\infty^* = S_\herd.
\end{equation}
    \item[(ii)]
    If $\alpha \in (\overline{\alpha},1]$, then
\begin{equation}
\label{eq325}
\lim_{D\to + \infty} S_\infty^*
= S_\infty(\alpha\mathbf 1)
<S_\herd.
\end{equation}
\end{enumerate}
\end{theorem}

In accordance with the notations introduced before, $\alpha\mathbf 1\equiv\alpha$ on $[0,+\infty)$, and $S_\infty(\alpha\mathbf 1)$ is the limit of $S(t)$ when $t\to +\infty$, for the solution of \eqref{SIR} corresponding to  $u=\alpha\mathbf 1$.

Theorem \ref{th3} establishes that, provided that the lockdown is sufficiently strong (more precisely, that $\alpha\leqslant \overline{\alpha}$), then long enough lockdown stops the disease propagation arbitrarily close after passing the herd immunity level.
On the contrary, if the lockdown is too moderate ($\alpha> \overline{\alpha}$), the power of such an action is intrinsically limited.
This phenomenon is clearly apparent in the simulations provided in Section \ref{se3}.

\begin{proof}[Proof of Theorem \ref{th3}]

One sees easily that $\overline\alpha >0$, due to the fact that $S_0>S_\herd$.
 On the other hand,
$$
\overline\alpha < \frac{S_\herd}{S_0-S_\herd}(\ln S_0 - \ln S_\herd)
= \frac{1}{S_0/S_\herd-1}\ln S_0/S_\herd
< 1.
$$

Assume now that $\alpha\leqslant \overline{\alpha}$.
\cite[Theorem 1]{Bliman:2020aa} establishes that, for any $\eps>0$, there exist $D>0$ and $u\in\cU_{\alpha, 0, D}$ such that  $S_\infty(u)\in [S_\herd-\eps,S_\herd]$.
As $S_\infty^* \geq S_\infty(u)$, this shows that
$$
    \limsup_{D\to + \infty} S_\infty^* \geq S_\herd.
$$
Due to the fact that $S_\infty^*$ is increasing with respect to $D$, as demonstrated by Theorem \ref{th2}, and that $S_\infty^* \leq S_\herd$ for any $D$, one gets \eqref{eq324}.

Suppose now $\alpha> \overline{\alpha}$.
In such conditions, \cite[Theorem 1]{Bliman:2020aa} shows that, for any $D>0$ and $u\in\cU_{\alpha, 0, D}$,
$$
S_\infty(u) \leq 
S_\infty(\alpha\mathbf 1) <S_\herd,
$$
so that
$$
\limsup_{D\to + \infty} S_\infty^* \leq S_\infty(\alpha\mathbf 1).
$$
On the other hand, the value of $S_\infty^*$ increases with $D$ (Theorem \ref{th2}), while $S_\infty(\alpha\mathbf 1)$ is the limit of $S_\infty(\alpha \mathds{1}_{[0,D]})$ for $D\to +\infty$.
This yields \eqref{eq325} and achieves the proof of Theorem \ref{th3}.
\end{proof}

\section{Numerical illustrations}
\label{se3}

We show in this Section the results of several numerical tests.
The algorithms designed to solve Problem \eqref{OCP2} are shown in \ref{se6}.
A case study is first presented in Section \ref{se31}, based on estimated conditions of circulation of the SARS‑CoV‑2 in France before and during the confinement enforced between March 17th and May 11th, 2020.
This example is chosen merely for its illustrative value, without claiming to a realistic description of the outburst.

The results provided and commented in Section \ref{se32} give a broader view.
They show the maximal final size reduction that may be obtained for a comprehensive set of basic reproduction numbers $\cR_0$, and for various realistic values of the maximal lockdown intensity $\alpha$ and duration $D$.

\subsection{Optimal lockdown in conditions of Covid-19 circulation in France, March--May 2020}
\label{se31}

The parameters used in the simulations of the present section are given in Table \ref{tab:value}.
We assume that, on the total number $N=\num{6.7e7}$ of individuals corresponding to the French population, there were initially no recovered individuals ($R_0=0$).
The initial number of infected individuals is taken equal to $1000$, a level crossed on March 8th \cite{Worldometer:aa}, so that $I_0=\num{1e3}/\num{6.7e7} \approx \num{1.49e-5}$.
Estimates of the infection rate $\beta$, of the recovery rate $\gamma$ and of the containment coefficient $\alpha_{\lock}$ in France between March 17th and May 11th 2020, are borrowed from \cite{Salje:2020aa}.
They yield the following values for the basic reproduction number and the herd immunity:
$$
\cR_0 \approx 2.9,\qquad  S_\herd\approx 0.34.
$$
With the initial conditions chosen here, the critical lockdown intensity defined in \eqref{eq323} is
$$
\overline\alpha \approx 0.56.
$$

\begin{table}[ht]
\centering
\begin{tabular}{|c|l|c|c|}
\hline
\textit{Parameter}&\textit{Name}&\textit{Value}\\\hline
 $\beta$& Infection rate &0.29 day$^{-1}$ \\\hline
 $\gamma$& Recovery rate &0.1 day$^{-1}$\\  \hline
     $\alpha_{\lock}$& Lockdown level (France, March-May 2020)&0.231 \\\hline
   $S_0$&Initial proportion of susceptible cases&$1-I_0$  \\ \hline
   $I_0$&Initial proportion of  infected cases & $\num{1.49e-5}$  \\ \hline
    $R_0$&Initial proportion of removed cases&0  \\ \hline
\end{tabular}
\caption{Value of the parameters used in the simulations for system \eqref{SIR} (see \cite{Salje:2020aa})}
\label{tab:value}
\end{table}

%

The optimal solution $(S^{*},I^{*},R^{*},u^*)$ of Problem \eqref{OCP2} for a containment duration of $30$ days (top), 60 days (middle) and 90 days (bottom) is shown in Fig.\ \ref{fig:u opt alpha=0}, when total lockdown is allowed ($\alpha=0$).
The evolution of the proportions of susceptible, infected and removed cases is shown on the left, the optimal control on the right.
The optimal dates for starting the enforcement are given in Table \ref{ta11}, together with the optimal asymptotic proportion of susceptible cases.

\begin{table}[h]
\begin{center}
\begin{tabular}{|c|c|c|c|}
\hline
$D$ & $T^*$ & $S_\infty^*$ & $S_\infty^*/S_\herd$\\
\hline\hline
No lockdown & --- &$0.0668$ &$0.194$ \\
\hline
$30$ days & $T^*=74.3$ days (May 21st) & $0.255$ & $0.739$\\
\hline
$60$ days &$T^*=74.3$ days (May 21st) & $0.323$&$0.937$ \\
\hline
$90$ days &$T^*=74.3$ days (May 21st) &$0.340$ &$0.985$ \\
\hline
\end{tabular}
\caption{Characteristics of the optimal solutions computed with the parameters of Table \ref{tab:value}, with lockdown intensity $\alpha=0$ and duration $D=0$ (no lockdown), $30, 60$ and $90$ days.
The starting dates are computed from the epidemic initial time on March 8th, where the cumulative number of infected exceeded 1000 cases.
See the curves in Figure \ref{fig:u opt alpha=0}, and explanations in text.}
\label{ta11}
\end{center}
\end{table}

\begin{table}[h]
\begin{center}
\begin{tabular}{|c|c|c|c|}
\hline
$D$ & $T^*$ & $S_\infty^*$ & $S_\infty^*/S_\herd$\\
\hline\hline
No lockdown & --- &$0.0668$ &$0.194$ \\
\hline
$30$ days & $T^*=72.1$ days (May 19th) & $0.222$ &$0.644$ \\
\hline
$60$ days & $T^*=71.5$ days (May 18th) & $0.302$ &$0.875$ \\
\hline
$90$ days & $T^*=71.3$ days (May 18th) & $0.331$ &$0.959$ \\
\hline
\end{tabular}
\caption{Similar to Table \ref{ta11}, with lockdown intensity $\alpha=\alpha_{\lock} \approx 0.231$.
See corresponding curves in Figure \ref{fig:compar sol}.}
\label{ta12}
\end{center}
\end{table}

As unveiled by close observation, one recovers the fact, established in Theorem \ref{th:T}, that $S(T^*)=S_\herd$: when $\alpha=0$, the optimal confinement starts exactly when the herd immunity threshold is crossed.
Also, the optimal value $S_\infty^*$ is larger when $D$ is larger (Theorem \ref{th2}), and it is known from Theorem \ref{th3} that this value converges towards $S_\herd$ when $D$ goes to infinity.
 It is indeed already indistinguishable from this value for $D=60$ and $90$ days.

Fig.\ \ref{fig:compar sol} shows the same numerical experiments than Fig.\ \ref{fig:u opt alpha=0}, with $\alpha =\alpha_{\lock} \approx 0.231 < \overline\alpha \approx 0.56$.
Optimal starting dates and asymptotic proportions of susceptible are given in Table \ref{ta12}.
The results are qualitatively similar to Fig.\ \ref{fig:u opt alpha=0}.
One sees that the lockdown begins earlier in the previous case, and the achieved $S_\infty^*$ are smaller.

The optimal starting dates given by the numerical resolution constitute an evident difference with the effective  implementation that took place during the Spring 2020 epidemic outburst: they are located in May, essentially at the time when, after two months of lockdown, first relaxation of the measures were introduced!
This should not be a surprise: the rationale behind this policy was not aimed at reaching herd immunity, but at reducing infections, in order to avoid overwhelming health systems and to be able to implement contact tracing on a tractable scale.
 On the contrary, the results in Fig.\ \ref{fig:u opt alpha=0} and \ref{fig:compar sol} show a peak of infected cases almost equal to 30\% of the population ---about twenty million people---, demonstrating that the strategy consisting of reaching herd immunity without considering other factors would not be sustainable, even if achieved under the optimal policy analyzed here.

\subsection{Maximal final size reduction under given epidemic and lockdown conditions}
\label{se32}

Once the optimal solution $u^*$ is computed, one may easily determine numerically, thanks to Lemma~\ref{co1} below, the optimal value $S_{\infty}^*$, by solving the equation
\begin{equation*}
S^*(T^*+D) + I^*(T^*+D) - \frac{\gamma}{\beta} \ln S^*(T^*+D)=
S_{\infty}^*- \frac{\gamma}{\beta} \ln S_{\infty}^*,
\end{equation*}
where, as said before, $(S^*, I^*)$ is the optimal solution.

Taking advantage of this principle, one may compute the optimal final size reduction corresponding to any epidemic and lockdown conditions.
To fix the ideas, the value $\gamma=0.1$ day$^{-1}$ is considered in this Section, corresponding to a mean recovery time of 10 days.
The general case is obtained by scaling: for \eqref{SIR} with different $\gamma$, the values of $S_\infty^*$ and $T^*$ are obtained as
$$
S_\infty^* = S_\infty^{'*},\qquad T^* = \frac{0.1}{\gamma} T^{'*}\ \text{day},
$$
where $S_\infty^{'*}$, $T^{'*}$ are the optimal cost and starting date obtained for the normalized system defined by the parameters $\gamma':= 0.1\ \text{day}^{-1}$ and
$$
\beta':= \frac{\beta}{\gamma} 0.1\ \text{day}^{-1},\qquad
\alpha' := \alpha,\qquad
D':= \frac{\gamma}{0.1} D\ \text{day}
$$
(in such a way that $\cR_0=\beta'/\gamma' = \beta/\gamma$).


\begin{figure}[H]
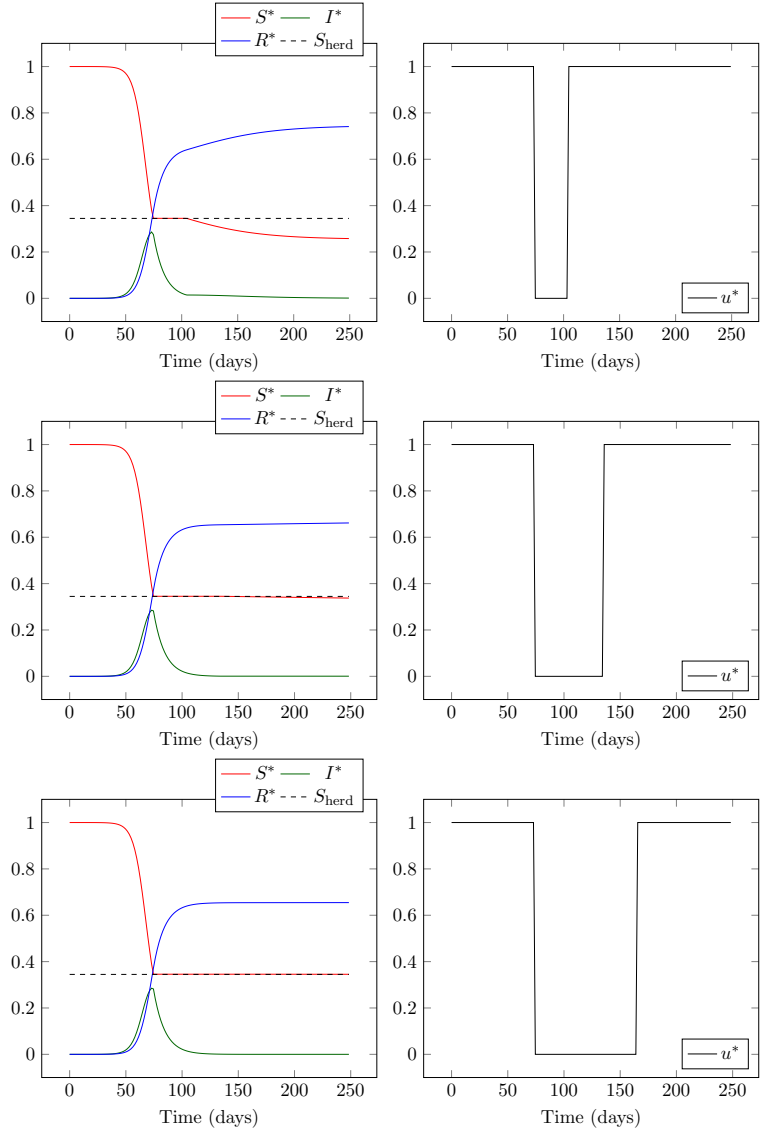
 
\begin{center}

\end{center}
\caption{
Optimal solution $(S^{*},I^{*},R^{*},u^*)$ to Problem \eqref{OCP2} computed for $\alpha=0.0$, $D=30, 60$ and $90$ days.
See numerical values in Table \ref{ta11} and comments in the text.
\label{fig:u opt alpha=0}}
\end{figure}

\begin{figure}[H]
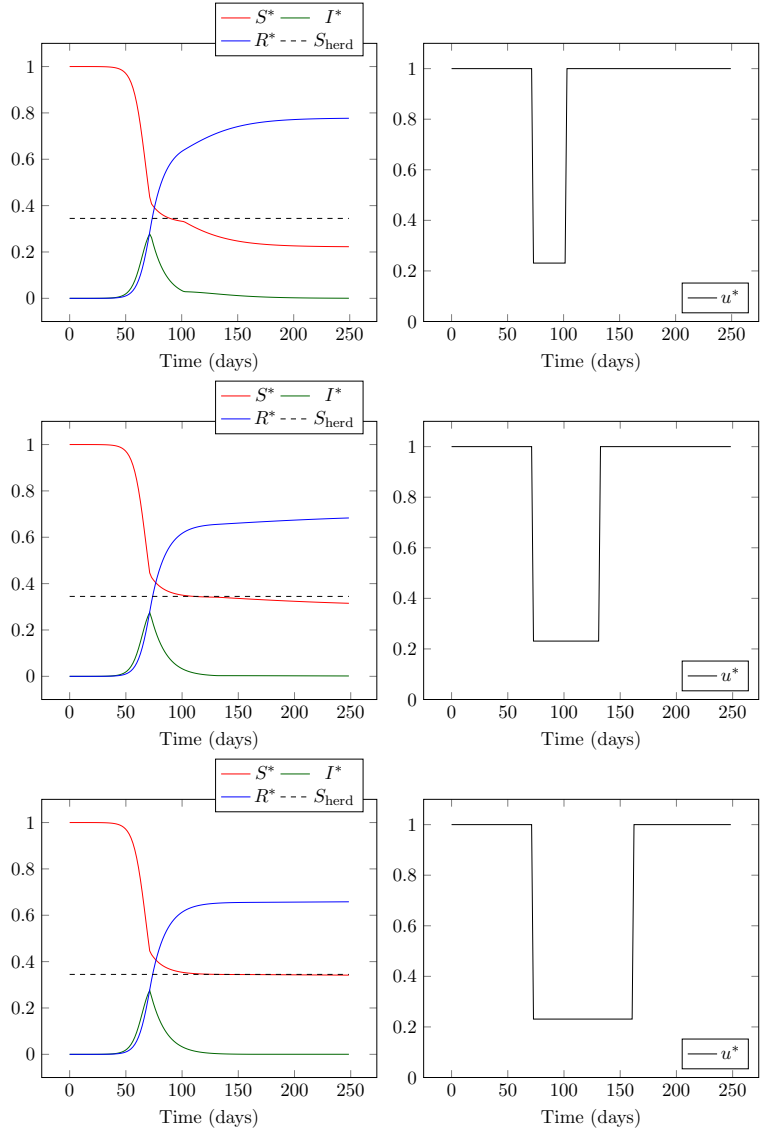
 
\begin{center}
%

\end{center}
\caption{Same than Figure \ref{fig:u opt alpha=0}, with $\alpha=\alpha_{\lock} \approx 0.231$.
Numerical values are provided in Table \ref{ta12}.
\label{fig:compar sol}}
\end{figure}

Fig.~\ref{fig:Sinf} shows, for an assortment of values of $\cR_0$ ranging from $1.5$ to $10$, the optimal final size value $S_{\infty}^*/S_\herd$ as a function of the duration $D$, for several values of the lockdown intensity $\alpha$ ranging from 0 to $0.8$.
 For $D$ tending to 0, all curves meet at a common value that corresponds to the final size attained in absence of lockdown.
One observes that the optimal value $S_{\infty}^*/S_\herd$ increases as a function of the lockdown duration $D$, and decreases as a function of its maximal intensity $\alpha$, as announced in Theorem \ref{th2}.
For $0<\alpha<\overline\alpha$, the optimal value $S_{\infty}^*$ converges towards $S_\herd$ (the best value one can expect) when $D$ increases indefinitely; while for $\overline\alpha<\alpha \leq 1$, the optimal value is strictly smaller, and decreases with respect to $\alpha$, as predicted by Theorem \ref{th3}.
Observe that the value of $\overline\alpha$ decreases when $\cR_0$ increases, making social distancing less efficient for diseases with larger basic reproduction number.
One sees that, depending upon the initial conditions, the optimal lockdown policy may induce a significant increase of the final size.

For the same assortment of values of $\cR_0$, we represent in Fig.\ \ref{fig:T} the dependency of $S_{\infty}^*/S_\herd$ with respect to the parameter $\alpha$, for values of $D$  corresponding to 1, 2, 4 and 8 months. 
For $\alpha$ tending to 1, the value of $S_{\infty}^*/S_\herd$ goes to the value achieved without control.

In the same way that Fig.\ \ref{fig:Sinf} and \ref{fig:T} revealed the dependence of $S_{\infty}^*/S_\herd$ respectively upon $D$ and $\alpha$, Fig.\ \ref{fig:T star D} and \ref{fig:T star alpha} show the dependence of $T^*$ with respect to these parameters.
Fig.\ \ref{fig:T star D} shows the variation of $T^*$ with respect to $D$, for the same values of $\alpha$ than Fig.\ \ref{fig:Sinf}.
The value of $T^*$ decreases as a function of $D$ and of $\alpha$.
One observes that for $D$ close to 0, the optimal intervention begins at the time where herd immunity is crossed, for every value of $\alpha$.
It converges to a positive limit when $\alpha>\overline\alpha$, while it converges to 0 when $\alpha<\overline\alpha$.
Notice that the value of $T^*$ highly depends upon $\cR_0$, ranging from more than 200 days for $\cR_0=1.5$ to less than 10 days for $\cR_0=10$.

Fig.\ \ref{fig:T star alpha} shows the variation of $T^*$ with respect to $\alpha$, for the same values of $D$ than Fig.\ \ref{fig:T}.
The value of $T^*$ also decreases with respect to $\alpha$.
When $\alpha=0$, the optimal starting point is at the crossing of the immunity threshold.

\section{Proof of Theorem \ref{th:T}}
\label{se2}

The proof is organized as follows.
We first recall in Section \ref{se20} results obtained in \cite{Bliman:2020aa} for the optimal control problem considered on intervention intervals of the type $[0,D]$, $D>0$.
Using these results, one shows in Section \ref{se21} that any solution of problem \eqref{OCP2} is of type \eqref{eq55} and may be determined by solving a 2D optimization problem.
It is subsequently shown in Section \ref{se22} that the latter problem may be simplified to a 1D optimization problem, whose study is achieved in Section \ref{se23}.
Last, the property on the limit of $T^*$ is demonstrated in Section \ref{se24}.

\subsection{Optimal control on a finite horizon $[0,D]$}
 \label{se20}

After introducing some notations, we recall here optimal control results from \cite{Bliman:2020aa}.

\begin{figure}[H]
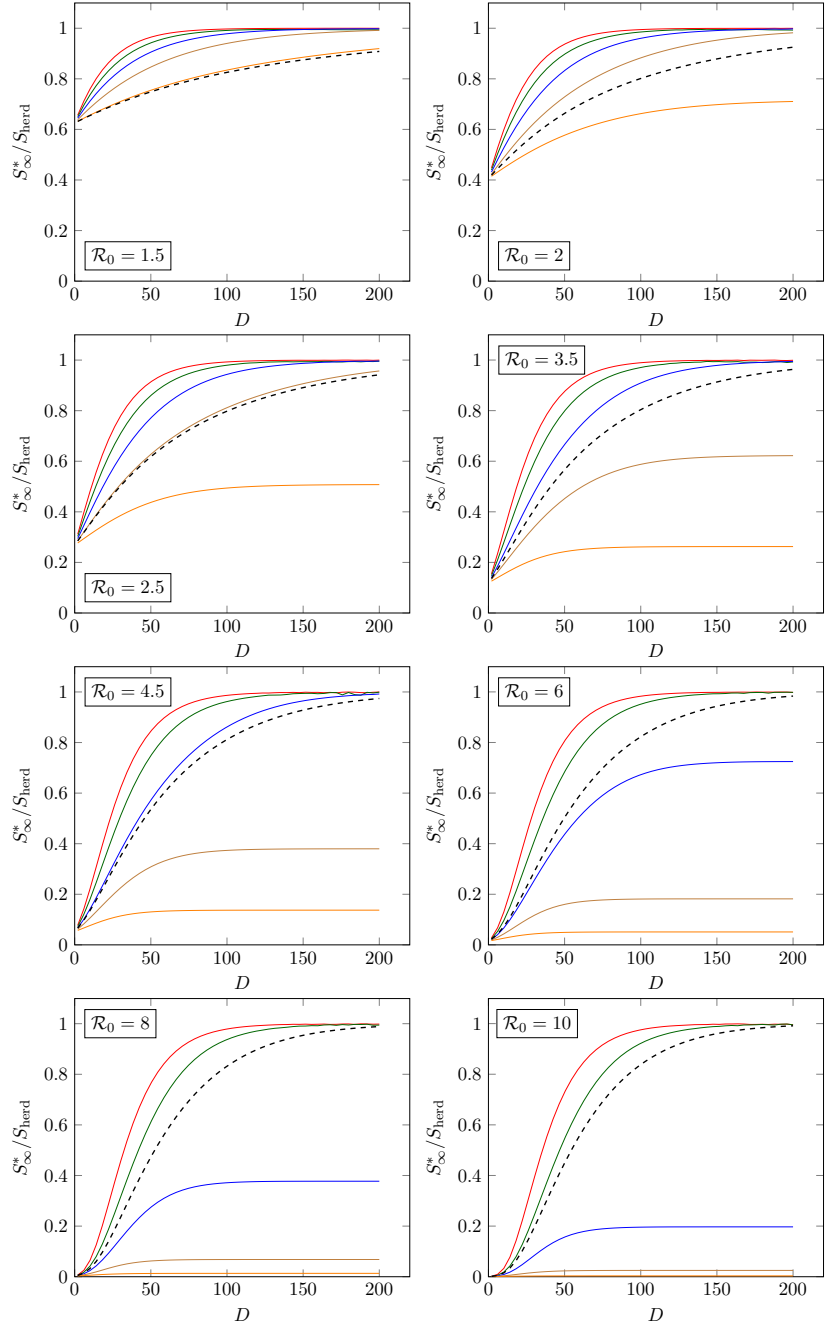
 
\begin{center}

\end{center}
\caption{
Graph of $S_{\infty}^*/S_\herd$  for Problem~\eqref{OCP2} 
as a function of $D$, for $\alpha=\{0.0$ (\textcolor{red}{\large\bf --}), 
$0.2$ (\textcolor{black!60!green}{\large\bf --}), 
$0.4$ (\textcolor{blue}{\large\bf --}), 
$0.6$ (\textcolor{brown}{\large\bf --}), 
$0.8$ (\textcolor{orange}{\large\bf --}), 
$\overline\alpha$ (\textcolor{black}{\large\bf - -})$\}$
  and $\mathcal{R}_0=\{1.5, 2, 2.5, 3.5, 4.5 ,6, 8, 10\}$.
\label{fig:Sinf}}
\end{figure}

\begin{figure}[H] 
\begin{center}

%

\end{center}

\caption{
Graph of $S_{\infty}^*/S_\herd$  for Problem~\eqref{OCP2}  
as a function of $\alpha$, 
for $D\in\{30$ (\textcolor{red}{\large\bf --}), 
$60$ (\textcolor{black!60!green}{\large\bf --}), 
$120$ (\textcolor{blue}{\large\bf --}), 
$240$ (\textcolor{black}{\large\bf --})$\}$ and $\mathcal{R}_0=\{1.5, 2, 2.5, 3.5, 4.5 ,6, 8, 10\}$.
\label{fig:T}}
\end{figure}

\begin{figure}[H]
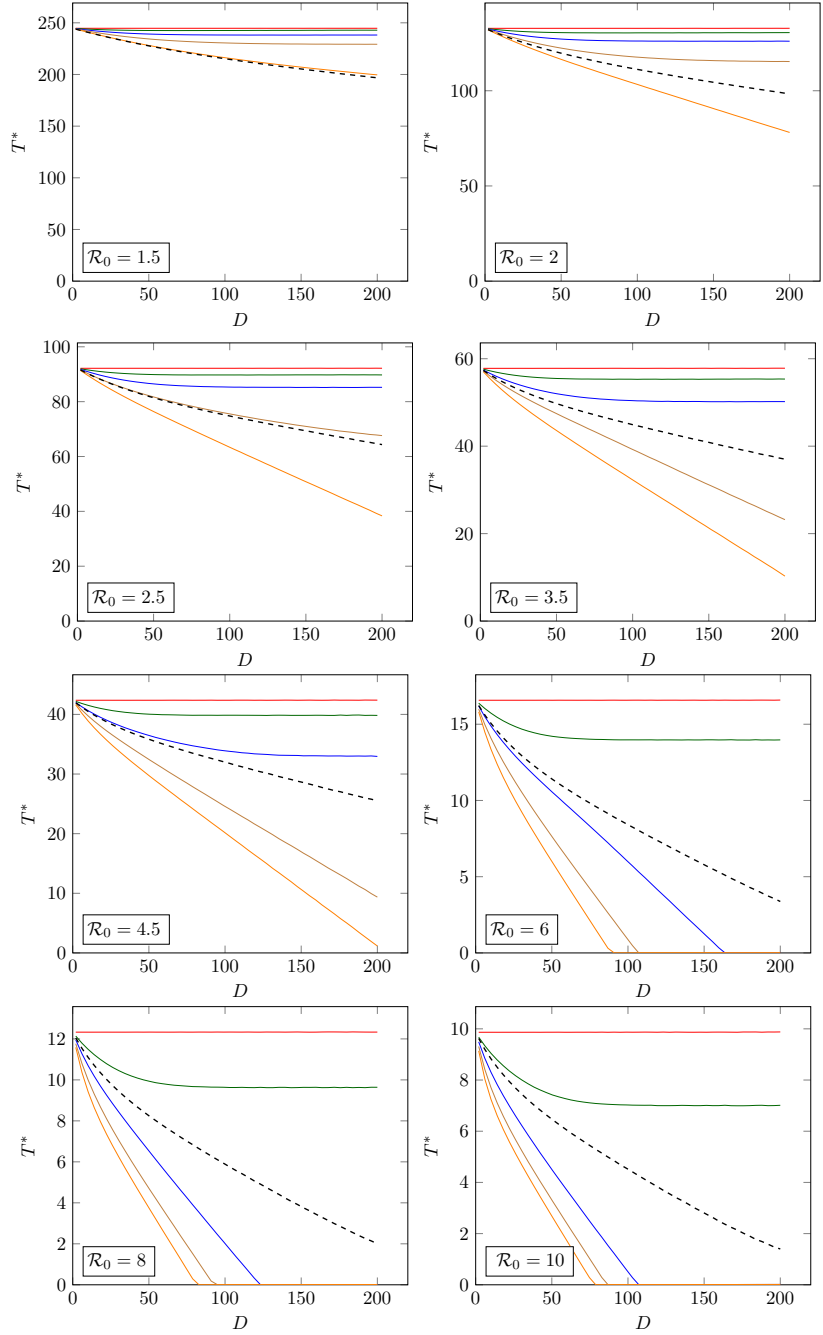
 
\begin{center}
%
\end{center}
\caption{
Graph of $T^*$  for Problem~\eqref{OCP2} 
as a function of $D$, for $\alpha=\{0.0$ (\textcolor{red}{\large\bf --}), 
$0.2$ (\textcolor{black!60!green}{\large\bf --}), 
$0.4$ (\textcolor{blue}{\large\bf --}), 
$0.6$ (\textcolor{brown}{\large\bf --}), 
$0.8$ (\textcolor{orange}{\large\bf --}), 
$\overline\alpha$ (\textcolor{black}{\large\bf - -})$\}$
  and $\mathcal{R}_0=\{1.5, 2, 2.5, 3.5, 4.5 ,6, 8, 10\}$.
\label{fig:T star D}}
\end{figure}

\begin{figure}[H]
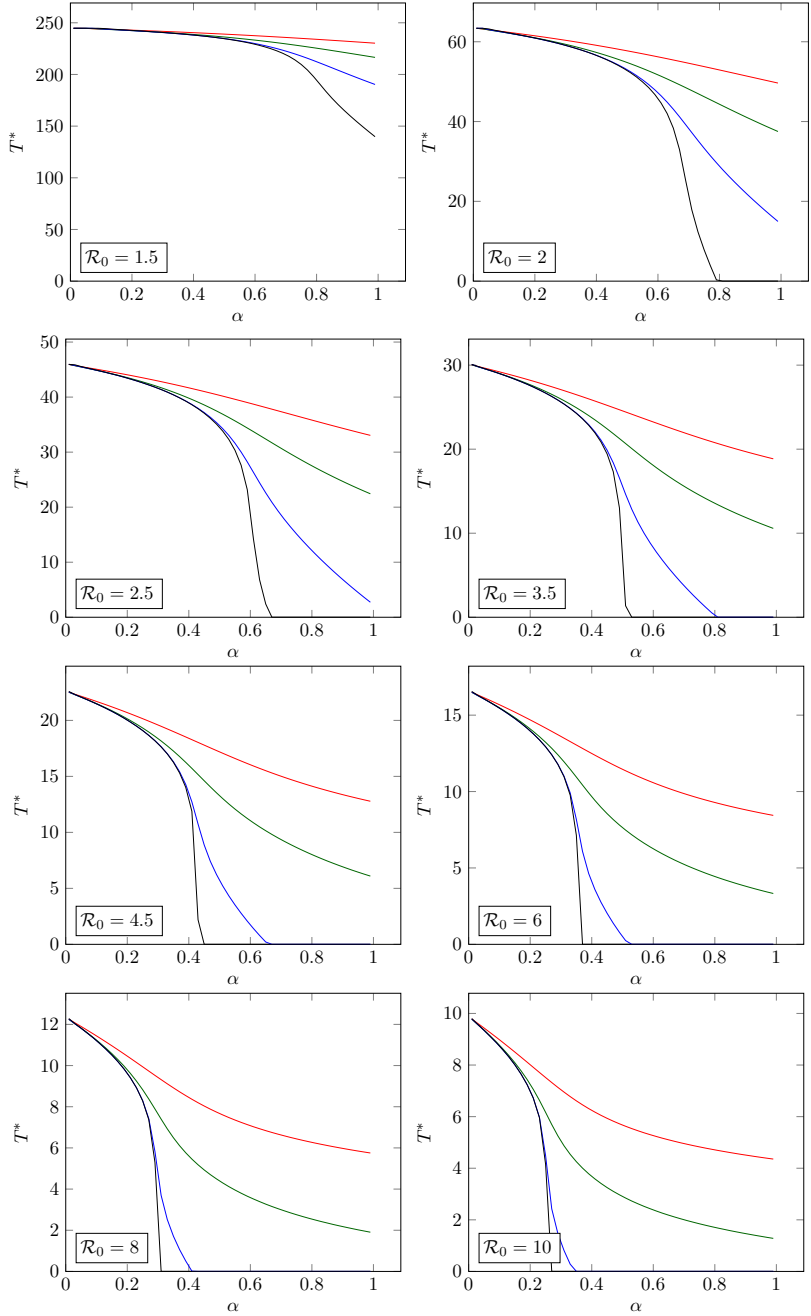
 
\begin{center}

%

\end{center}

\caption{
Graph of $T^*$  for Problem~\eqref{OCP2}  
as a function of $\alpha$, 
for $D\in\{30$ (\textcolor{red}{\large\bf --}), 
$60$ (\textcolor{black!60!green}{\large\bf --}), 
$120$ (\textcolor{blue}{\large\bf --}), 
$240$ (\textcolor{black}{\large\bf --})$\}$ and $\mathcal{R}_0=\{1.5, 2, 2.5, 3.5, 4.5 ,6, 8, 10\}$.
\label{fig:T star alpha}}
\end{figure}

For any $0\leq t\leq t'$, any input $u$,
and any admissible initial condition $X_0=(S_0,I_0)$, one denotes
$$
X(t',t;X_0;u) := (S(t',t;X_0;u),I(t',t;X_0;u))
$$
the value at time $t'$ of the solution of \eqref{SIR} departing at time $t$ from $X_0$, with the control input $u$.

This extended notation will be simplified when clear from the context.

Also, introduce the function $\Phi_\cR $ defined for any $\cR>0$ by:
\begin{equation}
\label{eq40}
\Phi_\cR :\RR_+^*\times \RR_+\ni (S,I) \longmapsto S + I - \frac{1}{\cR} \ln S.
\end{equation}
An important property is now given, which allows to define scalar quantities invariant along the trajectories.
See details in \cite[Lemma 3.1]{Bliman:2020aa}.
\begin{lemma}
\label{le1}
For any $u\in L^\infty([0,+\infty),[0,1])$ and $\gamma\in\RR$, one has
\begin{equation}
\label{eq34}
\frac{d}{dt} \left[
\Phi_\cR(S(t),I(t))
\right]
= \left(
\frac{\beta}{\cR} u(t) - \gamma
\right) I(t)
\end{equation}
along any trajectory of system \eqref{SIR}.
In particular, if $u$ is constant on a non-empty, possibly unbounded, interval, then the function $t\mapsto \Phi_{\cR_0 u}(S(t),I(t))$ is constant on this interval along any trajectory of system~\eqref{SIR}.
\end{lemma}

 The proof of Lemma \ref{le1} is straightforward and may be found in \cite{Bliman:2020aa}.
This result allows to characterize the epidemic final size resulting from the use of an input control in $\cU_{\alpha, T,T'}$,  as stated now, see details and proof in \cite[Lemma 3.2]{Bliman:2020aa}.

\begin{lemma}
\label{co1}
Let $0\leq T < T'$ and $u\in\cU_{\alpha, T,T'}$.
For any trajectory of \eqref{SIR}, 
$S_\infty(u)$ is the unique solution in $[0,S_\herd]$
of the equation
\begin{equation*}
\Phi_{\cR_0}(S_\infty(u),0) = \Phi_{\cR_0}(X(T',0;X_0;u)),
\end{equation*}
where $\Phi_{\cR_0}$ is given by \eqref{eq40}.
\end{lemma}

Due to the fact that any control input $u\in\cU_{\alpha, T,T+D}$ is equals to 1 on $[T+D,+\infty)$, the map $t\mapsto \Phi_{\cR_0}(X(t,0;X_0;u))$ is constant on that interval.
On the other hand, the map $S\mapsto \Phi_{\cR_0}(S,0)$ is {\em decreasing} on the interval $[0,S_\herd]$, so arguing as in \cite{Bliman:2020aa}, one deduces that solving problem \eqref{OCP2} is equivalent to solve
\begin{equation}
\label{OCP3}
\inf_{T\geq 0}\ \inf_{u\in\cU_{\alpha, T,T+D}} \Phi_{\cR_0}(X(T+D,0;X_0;u)).
\end{equation}
This property is central to our approach: it transforms \eqref{OCP2}, which consists in maximizing the limit of $S$ at infinity, into an optimal control problem on a {\em finite time horizon}.
This reduction procedure is at the basis of the arguments in \cite{Ketcheson:2020aa} and \cite{Bliman:2020aa}.
Based on the latter one obtains the following result.
See details in \cite[Theorem 2.3]{Bliman:2020aa}.

\begin{theorem}
\label{th1}
For any $\alpha\in [0,1)$ and $D>0$, the optimal control problem
\begin{equation}
\label{OCP}
\sup_{u\in\cU_{\alpha, 0, D}} S_\infty(u)
\end{equation}
admits a unique solution.
Moreover the optimal control, denoted $u^*_{0,D}(X_0)$, is equal to the function $u_{T_0^*,D}$ defined in \eqref{eq55}, for some uniquely defined $T_0^*\in[0, D)$.
\end{theorem}

Theorem \ref{th1} shows that the optimal control is bang-bang with at most two switches: a first one possibly at some time $T_0^*\in [0, D)$, and a second one at time $D$.
Moreover, the value of $T_0^*$ appears as the result of a 1D optimization problem, as
$$
\sup_{u\in\cU_{\alpha, 0, D}} S_\infty(u) = \sup_{T_0^*\in[0, D)} S_\infty(u_{T_0^*,D}).
$$

\subsection{Reduction to a 2D optimization problem}
 \label{se21}

Consider now, for any $0\leq T< T'$, the general problem 
\begin{equation}
\tag{$\mathcal{P}_{\alpha,T,T'}$}
\label{OCQ}
\inf_{u\in\cU_{\alpha, T,T'}} \Phi_{\cR_0}(X(T',0;X_0;u)).
\end{equation}
The problem  ($\mathcal{P}_{\alpha,0,D}$) is equivalent to \eqref{OCP}, which is the subject of Theorem \ref{th1}.
The following result shows that problem \eqref{OCQ}, under its equivalent form \eqref{OCP3}, benefits from this result.
\begin{proposition}
\label{pr1}
Let $0\leq T < T'$.
There exists a unique optimal control $u^*_{T,T'}(X_0)$ $ \in\cU_{\alpha, T,T'}$ for problem \eqref{OCQ}
and it verifies:
\begin{equation}
\label{eq6}
u^*_{T,T'}(X_0)(t) = u^*_{0,T'-T}(X(T,0;X_0; \mathbf 1))(t - T), \quad t\in [T,T'].
\end{equation}
\end{proposition}

Formula \eqref{eq6} and the fact that $u^*_{T,T'}(X_0) \in\cU_{\alpha, T,T'}$ imply that $u^*_{T,T'}(X_0)$ is equal to 1 on $[0,T]\cup [T',+\infty)$, so this function is uniquely defined on the whole $[0,+\infty)$ by the statement.
Proposition \ref{pr1} says that the optimal control for problem ($\mathcal{P}_{\alpha,T,T'}$) with initial value $X_0$ is equal to the optimal control for problem ($\mathcal{P}_{\alpha,0,T-T'}$) with initial value $X(T,0;X_0; \mathbf 1)$, delayed from the time duration $T$ and completed by 1 on the interval $[0,T]$.
In turn, the point $X(T,0;X_0; \mathbf 1)$ is the value at time $T$ of the solution of \eqref{SIR} departing at time 0 from $X_0$ with input equal to $\mathbf 1$.
Therefore, solving ($\mathcal{P}_{\alpha,T,T'}$) with initial condition $X_0$ amounts to solve ($\mathcal{P}_{\alpha,0,T'-T}$) with initial condition $X(T,0;X_0; \mathbf 1)$.

Before going further, let us prove the previous result.

\begin{proof}[Proof of Proposition \ref{pr1}]
One may define a canonical bijection $\cC\ :\ \cU_{\alpha, 0,T'-T} \to \cU_{\alpha, T,T'}$ by
$$
\cC(u)(t) = 1 \text{ if } t\in [0,T],\qquad \cC(u)(t) = u(t - T) \text{ if } t\in [T,+\infty)
$$
for any $u \in  \cU_{\alpha, 0,T'-T}$.
By the semi-group property deduced from the fact that system \eqref{SIR} is stationary, one has for any $u\in\cU_{\alpha, T,T'}$ and any $t\in [T,T']$,
\begin{equation*}
X(t,0;X_0;u) = X(t-T,0;X(T,0;X_0; \mathbf 1); \cC^{-1}(u)).
\end{equation*}
Applying this formula with $t=T'$ yields
$$
X(T',0;X_0;u) = X(T'-T,0;X(T,0;X_0; \mathbf 1); \cC^{-1}(u)).
$$
Therefore, for any $u\in\cU_{\alpha, T,T'}$,
\begin{equation*}
\Phi_{\cR_0}(X(T',0;X_0;u))
= \Phi_{\cR_0}(X(T'-T,0;X(T,0;X_0; \mathbf 1); \cC^{-1}(u))),
\end{equation*}
and this correspondence permits to achieve the demonstration.
\end{proof}

Using the qualitative properties of the solutions of problem \eqref{OCP3} recalled above, we deduce from Proposition \ref{pr1} that, for any $T\geq 0$, the problem
\begin{equation*}
 \inf_{u\in\cU_{\alpha, T,T+D}} \Phi_{\cR_0}(X(T+D,0;X_0;u))
\end{equation*}
admits a unique solution of the type $u_{T+\varepsilon_{T,D},T+D}$, for some $\varepsilon_{T,D} \in [0,D)$.
We thus have proved so far that problem \eqref{OCP2} is equivalent to the 2D optimization problem
\begin{equation*}
\inf_{T\geq 0} \inf_{\varepsilon\in [0,D)}\Phi_{\cR_0}(X(T+D,0;X_0;u_{T+\varepsilon,T+D})).
\end{equation*}

\subsection{Reduction to a 1D optimization problem}
\label{se22}

In this section, we further reduce the complexity of the optimal control problem under study.
We first show that the problem \eqref{OCP2} admits a solution.
This indeed amounts to show that no unbounded maximizing sequence of times $\{T_k\}_{k\in\NN}$ is to be found. 

\begin{proposition}
\label{pr2}
Problem \eqref{OCP2} admits at least one solution.
\end{proposition}
\begin{proof}
Consider $\bar X := (\bar S,\bar I)$ the solution associated to $u= \mathbf 1$, and let $\bar T$ be defined by $\bar S (\bar T) = S_\herd$.
For this value $\bar T$, define $X^{\bar T} := (S^{\bar T},I^{\bar T})$ as the solution to system  \eqref{SIR} associated to $u_{\bar T,\bar T+D}$.
Lemma \ref{le1} shows that
\begin{itemize}
\item
the map $t\mapsto \Phi_{\cR_0}(\bar X(t))$ is constant on $[0,+\infty)$;
\item
the map $t\mapsto \Phi_{\cR_0}(X^{\bar T}(t))$ is constant on $[0,\bar T]$ and on $[\bar T+D, +\infty)$;
\item
the value of $\Phi_{\cR_0}(X^{\bar T}(t))$ on $[0,\bar T]$ is {\em smaller} than the value on $[\bar T+D, +\infty)$, because \eqref{eq34} implies that this map cannot increase on $[\bar T,\bar T+D]$.
\end{itemize}
Therefore, the fact that $\Phi_{\cR_0}(\bar X(0)) = \Phi_{\cR_0}(X_0) = \Phi_{\cR_0}(X^{\bar T}(0))$ implies
$$\Phi_{\cR_0}(X^{\bar T}(\bar T+D)) < \Phi_{\cR_0}(\bar X(\bar T+D)),$$
and thus
$$
S_{\infty}(u_{\bar T,\bar T+D})-\frac{\gamma}{\beta}\ln(S_{\infty}(u_{\bar T,\bar T+D}))
< S_{\infty}({\mathbf 1})-\frac{\gamma}{\beta}\ln(S_{\infty}({\mathbf 1})).
$$

Since $\Phi_{\cR_0}(\cdot,0)$ is decreasing on $(0,S_\herd)$ and, by Lemma \ref{co1}, $S_{\infty}(u_{\bar T,\bar T+D}),$ $S_{\infty}({\mathbf 1})< S_\herd$, we deduce that
$$S_{\infty}(u_{\bar T,\bar T+D}) > S_{\infty}({\mathbf 1}).$$

There thus exists $T_1$ sufficiently large, so that
$$S_{\infty}({\mathbf 1}) < \bar S(T_1) < S_{\infty}(u_{\bar T,\bar T+D}).$$
Since $S$ decreases along every trajectory, for each $T>T_1$ and $\varepsilon\in(0,D)$, one has
$$S_{\infty}(u_{T+\varepsilon,T+D})< \bar S(T_1),$$
because $u_{T+\varepsilon,T+D} \equiv 1$ on $[0,T_1] \subset [0,T+\varepsilon]$. 
Therefore, one may thus restrict the search for optimal solutions of problem \eqref{OCP2} to those $(T,\varepsilon)$ that belong to the set $[0, T_1]\times [0,D]$.
We conclude by observing that the problem
\begin{equation}
\label{OCP2ter}
\inf_{T\in [0, T_1]} \inf_{\varepsilon\in [0,D]}\Phi_{\cR_0}(X(T+D,0;X_0;u_{T+\varepsilon,T+D})),
\end{equation}
which consists in optimizing a continuous function on a finite-dimensional compact set, admits a non-void set of solutions.
\end{proof}

We now show in the next result that every possible optimal solution for problem \eqref{OCP2ter} corresponds to $\varepsilon=0$.
In other terms, any optimal policy consists in applying the more intense lockdown intensity during a duration exactly equal to $D$, not less.

\begin{proposition}
\label{pr3}
Any solution of problem \eqref{OCP2} is of the type $u_{T,T+D}$ for some $T\geq 0$.
\end{proposition}

From Proposition \ref{pr3} one deduces that problem \eqref{OCP2} is equivalent to solving
\begin{equation}
\label{OCP2quar}
\inf_{T\in [0,\bar T]} \Phi_{\cR_0}(X(T+D,0;X_0;u_{T,T+D})).
\end{equation}
This achieves the announced reduction to a 1D optimization problem.

\begin{proof}[Proof of Proposition \ref{pr3}]
Assume by contradiction that $u_{T+\varepsilon,T+D}$ is solution to problem \eqref{OCP2} for some $\varepsilon>0$.
Then 
$u_{T+\varepsilon,T+D}\in \cU_{\alpha, T+\varepsilon,T+D} \cap \cU_{\alpha, T+\varepsilon,T+\varepsilon+D}$.
We know from Proposition \ref{pr2} and Theorem \ref{th1} that 
\begin{equation*}
 \inf_{u\in\cU_{\alpha, T+\varepsilon,T+\varepsilon+D}} \Phi_{\cR_0}(X(T+D,0;X_0;u))
\end{equation*}
admits a unique solution, which writes $u_{T+\varepsilon+\delta,T+\varepsilon+D}$ for some $\delta\geqslant 0$. Since 
$u_{T+\varepsilon,T+D}\neq u_{T+\varepsilon+\delta,T+\varepsilon+D}$, one has
$$ \Phi_{\cR_0}(X(T+D,0;X_0;u_{T+\varepsilon+\delta,T+\varepsilon+D}))< \Phi_{\cR_0}(X(T+D,0;X_0;u_{T+\varepsilon,T+D})).$$
This is in contradiction with the optimality of $u_{T+\varepsilon,T+D}$ for problem \eqref{OCP2}.
Therefore, $\varepsilon=0$ for any optimal control.
\end{proof}

\subsection{Solving the 1D optimization problem (\ref{OCP2quar})}
\label{se23}


We now achieve the demonstration of Theorem \ref{th:T}, through the study of problem \eqref{OCP2quar}, establishing in particular uniqueness of the optimum control.
Denote $T^*$ an optimal solution of \eqref{OCP2quar}.

\paragraph{Step 1: necessary first order optimality conditions}
Let $u=u_{T,T+D}$ be an optimal control for problem~\eqref{OCP2}.
Let us introduce the criterion $j$ given by
\[
j(T):=\Phi_{\mathcal{R}_0}(S^{T}(T+D),I^{T}(T+D))=I^{T}(T+D)+S^{T}(T+D) - \frac{\gamma}{\beta} \ln( S^{T}(T+D)),
\]
where, as done in the definition of $\psi$ in \eqref{eq:psi new},
$(S^{T},I^{T})$ is the solution corresponding to the control $u_{T,T+D}$.
For sake of simplicity, we will usually omit these subscripts in the sequel.
With this notation, \eqref{OCP2quar} simply writes
\begin{equation}\label{eq:opt j}
\inf_{T\geq 0} j(T).
\end{equation}
By using Lemma~\ref{co1}, one has for the solution $(S^{T},I^{T})$:
\begin{equation}\label{m1440}
\begin{array}{ll}
I(t)+S(t)- \frac{\gamma}{\beta} \ln S(t)=c_0 & \text{in }[0,T],\\
I(t)+S(t)-\frac{\gamma}{\alpha\beta} \ln S(t)=I(T)+S(T)-\frac{\gamma}{\alpha\beta} \ln S(T) & \text{in }[T,T+D],
\end{array}
\end{equation}
where $c_0:=I_0+S_0- \frac{\gamma}{\beta} \ln S_0$.
 Eliminating $I(t)$ from \eqref{SIR} with $u=u_{T,T+D}$, thanks to \eqref{m1440}, we infer that $S$ solves the system
\begin{subequations}
\begin{align}
& \dot S=-\beta S(c_0-S+ \frac{\gamma}{\beta} \ln S), \qquad \text{in }(0,T), \label{S1a}\\
& \dot S=-\alpha\beta S\left(c_0+\frac{\gamma}{\beta}\left(1-\frac{1}{\alpha} \right)\ln S(T)-S+\frac{\gamma}{\alpha\beta} \ln S\right), \quad \text{in }(T,T+D),  \label{S1b}
\end{align}
\end{subequations}
with the initial value $S(0)=S_0$.
Using \eqref{m1440}, one gets 
\begin{eqnarray*}
j(T) &=& I(T+D)+S(T+D)- \frac{\gamma}{\beta} \ln S(T+D)\\
&=& I(T)+S(T)-\frac{\gamma}{\alpha\beta} \ln S(T)+\frac{\gamma}{\beta}\left(\frac{1}{\alpha}-1\right) \ln S(T+D)\\
       &=& c_0+ \frac{\gamma}{\beta} \ln S(T)-\frac{\gamma}{\alpha\beta}\ln S(T)+\frac{\gamma}{\beta}\left(\frac{1}{\alpha}-1\right) \ln S(T+D),
\end{eqnarray*}
so that the cost function reads
\begin{equation*}
  j(T) = c_0+\frac{\gamma}{\beta}\left(\frac{1}{\alpha}-1\right) \ln \left(\frac{S^T(T+D)}{S^T(T)}\right).
\end{equation*}
We point out that this expression depends upon $T$ through the arguments $T$ and $T+D$ at which the function $S^T$ is considered;
but also through the value of the function $S^T$ itself, which depends upon $T$ through the input $u_{T,T+D}$. 
Special care is therefore needed to compute the derivative $j'$ of $j$ with respect to $T$.
This constitutes the subject of the following technical lemma, whose proof is postponed to the end of the section, for sake of clarity.
For simplicity, we denote in the sequel $\widehat{S(T+D)}$, $\widehat{S(T)}$ and $\widehat{S(t)}$ the derivatives of the functions $S^T(T+D)$, $S^T(T)$ and $S^T(t)$ with respect to $T$, that is:
\begin{gather*}
\widehat{S(T+D)} := \frac{\partial [S^T(T+D)]}{\partial T},\qquad \widehat{S(T)} := \frac{\partial [S^T(T)]}{\partial T},\\
\widehat{S(t)} := \frac{\partial [S^T(t)]}{\partial T},\qquad t\in (T,T+D).
\end{gather*}

\begin{lemma}\label{lem:HatS2}
The following formulas hold.
\begin{subequations}
\begin{gather}
\label{eq66a}
\hspace{-.6cm}
\widehat{S(T+D)}=\beta S^{T}(T+D)I^{T}(T+D)\left(-1+(\alpha-1)\gamma I^{T}(T)\int_{T}^{T+D}\frac{ds}{I^{T}(s)}\right),\\
\label{eq66b}
\widehat{S(T)}=-\beta S^T(T)I^T(T),\\
\label{eq66c}
\widehat{S(t)}=(\alpha-1)\beta S^T(t)I^T(t)\left(1+\gamma I^T(T)\int_{T}^{t}\frac{ds}{I^T(s)}\right),\quad
t\in (T,T+D).
\end{gather}
\end{subequations}
\end{lemma}
Thanks to identities \eqref{eq66a}-\eqref{eq66b}, one may compute
\begin{eqnarray*}
j'(T) 
 & = &
 \frac{\gamma}{\beta}\left(\frac{1}{\alpha}-1\right)\left(\frac{\widehat{S(T+D)}}{S(T+D)}-\frac{\widehat{S(T)}}{S(T)}\right)\\
 & = &
 \gamma\left(\frac{1}{\alpha}-1\right)\\&&\hspace*{1cm}\times\left(I(T+D)\left(-1+(\alpha-1)\gamma I(T)\int_{T}^{T+D}\frac{ds}{I(s)}\right)+I(T)\right) \\
 & = &
 \gamma\left(\frac{1}{\alpha}-1\right)I(T)\\&&\hspace*{2cm}\times\left(-\frac{I(T+D)}{I(T)}+(\alpha-1)\gamma \int_{T}^{T+D}\frac{I(T+D)}{I(s)}ds+1\right)  .
\end{eqnarray*}
We deduce that $j'(T)=0$ is equivalent to 
\begin{equation}\label{eq:cond psi}
\psi(T)=0,
\end{equation}
for the function $\psi$ defined in \eqref{eq:psi new}.

\paragraph{Step 2: Zeros of $j'$ and uniqueness of the optimal time $T^*$}
 Integrating the second equation in \eqref{SIR} for $u=u_{T,T+D}$, one has for any  $t\in (T,T+D)$,
  $I(t) = I(T+D) \exp\left(\int_{T+D}^t (\alpha\beta S(s)-\gamma)\,ds\right)$.
  Then, using the expression of $\psi$ in \eqref{eq:psi new},  it follows that
\begin{eqnarray*}
\psi(T)
& = &
- \exp\left(\int_{T}^{T+D} (\alpha\beta S(t)-\gamma)\,dt\right)\\
& &
 +(\alpha-1)\gamma \int_{T}^{T+D} \exp\left(\int^{T+D}_s (\alpha\beta S(t)-\gamma)\,dt\right)\,ds   +1.
\end{eqnarray*}
Introducing $\varphi(t):=\exp\left(\int_t^{T+D} (\alpha\beta S(s)-\gamma)\,ds\right)$ for $t\in [0,T+D]$,  the last expression writes simply
 \begin{multline*}
 \psi(T)
 = -\varphi(T)
  + (\alpha-1)\gamma \int_{T}^{T+D} \varphi(t)\,dt   +1\\
  = -\varphi(T)
  + (\alpha-1)\gamma \int_{0}^{D} \varphi(T+t)\,dt   +1  .
  \end{multline*}
  Differentiating the expressions in $\varphi$ with respect to $t$  and afterwards $\psi$  with respect to $T$ yields first
$$
\varphi'(t)
= \alpha\beta \left(S(T+D)-S(t)+\int_{t}^{T+D}  \widehat{S(s)}\,ds\right) \varphi(t),\qquad t\geq 0,
$$
and then
\begin{multline*}
\psi'(T) =
 - \alpha\beta\left(S(T+D)- S(T) +  \int_{T}^{T+D}  \widehat{S(s)}\,ds \right) \varphi(T)\\
    + (\alpha-1) \gamma\alpha\beta\int_{0}^{D} \left(S(T+D)-S(T+t)+\int_{T+t}^{T+D}  \widehat{S(s)}\,ds\right) \varphi(T+t)\, dt.
  \end{multline*}
On the one hand, $S$ decreases along the trajectory, so $S(T+D)-S(T+t) < 0$ for any $t\in [0,D)$.
On the other hand, $\widehat{S(t)} <0$, see formula \eqref{eq66c}.
One then deduces that both terms in the addition in the previous formula are positive.
The function $\psi$ is thus increasing on $(0,\infty)$.

\paragraph{Step 3: The case $\alpha>0$}

Assume now that $\alpha>0$.
Then, for any $T$ large enough in such a way that $S^T(T) < S_\herd$, one has
 $\alpha\beta S^T(t)-\gamma<\alpha\beta S_\herd-\gamma=\gamma(\alpha-1)$ for any $t\in(T,T+D)$,
 because the function $S^T$ is decreasing along every trajectory.
 For such a sufficiently large $T$, one has
\begin{equation}
\label{eq70}
\hspace{-.2cm}
 \varphi(t)
 = \exp\left(\int_t^{T+D} (\alpha\beta S(s)-\gamma)\,ds\right)
 < e^{\gamma(\alpha-1)(T+D-t)},
 \quad t\in (T,T+D],
\end{equation}
and thus
\begin{eqnarray*}
\psi(T)
& = &
-\varphi(T)
  + (\alpha-1)\gamma \int_{0}^{D} \varphi(T+t)\,dt   +1\\
 & > &
 -e^{\gamma(\alpha-1) D}+(\alpha-1)\gamma\int_0^{D}e^{(D-t)\gamma(\alpha-1)}dt+1=0.
\end{eqnarray*}
Therefore, the function $\psi$ being increasing, if $\psi(0)>0$, then \eqref{eq:cond psi} has no solution.
Equivalently there is no $T$ such that $j'(T)=0$, and thus $T^*=0$.
Conversely,  if  $\psi(0)\leq0$, then \eqref{eq:cond psi} admits a unique solution $T^*$, which is the unique critical point of $j$.
In the particular case where $\psi(0)=0$, one has $T^*=0$.
 
 \begin{remark}\label{rmk1}
 Notice that the function $j$ is decreasing on $(0,T^*)$ and increasing on $(T^*,\infty)$.
 This observation will be useful for the numerical implementation.
 \end{remark}

The fact that $S(T^*)> S_\herd$ if $T^*>0$ comes as a byproduct of the previous considerations.
Indeed, it has been shown that $\psi(T)>0$ if $S^T(T) < S_\herd$.
Therefore, if $T^*>0$, then $\psi(T^*)=0$ and $S(T^*)\geqslant S_\herd$.
Noticing that the inequality in \eqref{eq70} is strict for any $t\in (T,T+D)$ yields the strict inequality $S(T^*)> S_\herd$.

\paragraph{Step 4: The case $\alpha=0$}
In the case $\alpha=0$, the solution $S^T$ corresponding to $u_{T,T+D}$ is constant on $(T,T+D)$, and we deduce that 
\begin{align*}
j(T)&=I(T+D)+S(T)-\frac{\gamma}{\beta}\ln(S(T))=(e^{-\gamma D}-1)I(T)+c_0.
\end{align*}
We conclude using the fact that $I(T)$ is maximal when $S(T)=S_\herd$, therefore $T^*$ is such that $S(T^*)=S_\herd$.

%

To terminate the work done in Section \ref{se23}, it now remains to prove Lemma~\ref{lem:HatS2}.

\begin{proof}[Proof of Lemma~\ref{lem:HatS2}]
Using the notation $S^{T}$ previously defined, one has (see \eqref{S1b}) on $(T,T+D)$
\begin{equation*}
\dot S^{T}=-\alpha\beta  S^{T}\left(c_0+\frac{\gamma}{\beta }\left(1-\frac{1}{\alpha} \right)\ln (S^{T}(T))-S^{T}+\frac{\gamma}{\alpha\beta } \ln S^{T}\right),
\end{equation*}
and at time $T$, $S^{T}(T)$ is defined thanks to \eqref{S1a} by 
\begin{equation}\label{m008bis}
\int_{S_0}^{S^{T}(T)}\frac{dv}{\beta  v(c_0-v+\frac{\gamma}{\beta }\ln v)}=-T.
\end{equation}
By differentiating \eqref{m008bis} with respect to $T$, one infers
\[
\widehat{S^{T}(T)}=-\beta S^{T}(T)\left(c_0-S^{T}(T)+ \frac{\gamma}{\beta } \ln (S^{T}(T))\right)=-\beta S^{T}(T)I^{T}(T),
\]
that is \eqref{eq66b}.

Furthermore,  using \eqref{S1b}, one has
\begin{equation*}
\int_{S^{T}(T)}^{S^{T}(T+D)}\frac{dv}{v(c_0+\frac{\gamma}{\beta }\left(1-\frac{1}{\alpha} \right)\ln (S^{T}(T))-v+\frac{\gamma}{\alpha\beta }\ln v)}=-\alpha\beta D.
\end{equation*}
Differentiating this relation with respect to $T$ yields 
\begin{multline*}
\frac{\widehat{S^{T}(T+D)}}{S^{T}(T+D)(c_0+\frac{\gamma}{\beta }\left(1-\frac{1}{\alpha} \right)\ln S^{T}(T)-S^{T}(T+D)+\frac{\gamma}{\alpha\beta }\ln S^{T}(T+D))}\\
-\frac{\widehat{S^{T}(T)}}{S^{T}(T)(c_0+\frac{\gamma}{\beta }\ln S^{T}(T)-S^{T}(T))}
-\frac{\gamma}{\beta } \left(1-\frac{1}{\alpha}\right)\frac{\widehat{S^{T}(T)}}{S^{T}(T)}\\
\times\int_{S^{T}(T)}^{S^{T}(T+D)}\frac{dv}{v(c_0+\frac{\gamma}{\beta }\left(1-\frac{1}{\alpha} \right)\ln S^{T}(T)-v+\frac{\gamma}{\alpha\beta }\ln v)^2}=0.
\end{multline*}

Let us simplify this latter identity.
Observe first that, because of \eqref{m1440}, one has 
\begin{align*}
  c_0+\frac{\gamma}{\beta } \left(1-\frac{1}{\alpha}\right)\ln S^{T}(T)
  & = I^{T}(t)+S^{T}(t)-\frac{\gamma}{\alpha\beta } \ln S^{T}(t)
\end{align*}
for each $t\in(T,T+D)$. By using at the same time the change of variable $v=S(t)$ and the expression of $I^T(t)$, $t\in(T,T+D)$, extracted from this identity,
we infer that
\begin{eqnarray*}
\lefteqn{\int_{S^{T}(T)}^{S^{T}(T+D)} \frac{dv}{\beta v(c_0+\frac{\gamma}{\beta }\left(1-\frac{1}{\alpha} \right)\ln S^{T}(T)-v+\frac{\gamma}{\alpha\beta }\ln v)^2}}\\
& = &
\int_{T}^{T+D} \frac{1}{\beta S^{T}(s)(I^{T}(s))^2} \dot{S}^{T}(s) \,ds
= \int_{T}^{T+D} \frac{-\alpha S^{T}(s)I^{T}(s)}{S^{T}(s)(I^{T}(s))^2}\,ds\\
& = &
-\alpha \int_{T}^{T+D} \frac{ds}{I^{T}(s)}. 
\end{eqnarray*}
Combining all these facts leads to
  \begin{align*}
 0 = \frac{\widehat{S^{T}(T+D)}}{S^{T}(T+D) I^{T}(T+D)} + \beta + \gamma (1-\alpha)\beta I^{T}(T) \int_{T}^{T+D} \frac{ds}{I^{T}(s)},
  \end{align*}
and we arrive at \eqref{eq66a}.

Similar arguments allow for the computation of $\widehat{S(t)}$.
This achieves the proof of Lemma~\ref{lem:HatS2}.
\end{proof}


\subsection{Limit behaviour of $T^*$ when $I_0$ vanishes}
\label{se24}

To complete the demonstration of Theorem~\ref{th:T}, it now remains to prove the last property of the statement.
The following result is instrumental for this purpose.

\begin{lemma}
\label{le18}
Assume $S_0\in (S_\herd,1)$.
For any $\bar T\geq 0$, there exist $c>0$ and $\bar I_0>0$ such that
\begin{equation}
\label{eq887}
\forall I_0\in (0,\bar I_0),\quad \max\limits_{T\in [0,\bar T]} \psi(T) < -c <0,
\end{equation}
where $\psi$ defined in \eqref{eq:psi new} depends upon $I_0$ through the initial value of $(S^T,I^T)$.
\end{lemma}

Using the characterization (already demonstrated above) of $T^*$ given in Theorem \ref{th:T}, one deduces straightforwardly from Lemma \ref{le18} that
\[
\lim\limits_{I_0\searrow 0^+} T^* = +\infty.
\]

\begin{proof}[Proof of Lemma~\ref{le18}]
Let $T\geq 0$.
From the fact that $\dot I^T \leq (\beta S_0-\gamma) I^T$, $I^T(0)=I_0$, one deduces that
$0 \leq I^T(t) \leq I_0 e^{(\beta S_0 - \gamma) t}$, $t\in [0,T+D]$.

From this, one deduces that $|\dot S^T| = \beta u(t) S^TI^T \leq \eta(I_0) S^T$, where $\eta(I_0)$ represents, here and in the sequel, quantities that converge to 0 when $I_0$ vanishes, {\em uniformly on $[0,T+D]$ when they depend upon time $t$}.
Therefore,
$$
S^T(t) = S_0 + \eta(I_0),\qquad t\in [0,T+D].
$$

Define now $\omega:= \alpha\beta S_0 - \gamma$.
From the foregoing, one has
\begin{equation}
\label{eq888}
\begin{cases}
\dot I^T = (\beta S_0 - \gamma) I^T + I_0 \eta(I_0), & t \in [0,T],\\
\dot I^T = \omega I^T + I_0 \eta(I_0), & t \in [T,T+D].
\end{cases}
\end{equation}
By integration one deduces from \eqref{eq888} that $I^T(t) = I_0 e^{(\beta S_0 - \gamma) t} + I_0 \eta(I_0)$ for any $t \in [0,T]$, and in particular that $I^T(T) = I_0 (e^{(\beta S_0 - \gamma) T} + \eta(I_0))$.

Assume first $\omega\neq 0$, then integration of the second formula in \eqref{eq888} yields
\begin{eqnarray}
I^T(t)
& = &
\nonumber
I^T(T) e^{\omega (t-T)} + I_0 \eta(I_0)\\
& = &
\label{eq890}
I_0
\left(
e^{(\beta S_0 - \gamma) T} e^{\omega (t-T)} + \eta(I_0)
\right),\qquad t \in [T,T+D].
\end{eqnarray}
Using \eqref{eq890} to compute the value $\psi(T)$ in \eqref{eq:psi new} then shows that
\begin{eqnarray}
\psi(T)
& = &
\nonumber
-e^{\omega D} + (\alpha-1)\gamma \frac{e^{\omega D} - 1}{\omega} + 1 + \eta(I_0)\\
& = &
\label{eq889}
\left(
\frac{(\alpha-1)\gamma}{\omega} -1 
\right) \left(
e^{\omega D} - 1
\right) + \eta(I_0). 
\end{eqnarray}
If $\omega >0$, then due to the fact that $\alpha-1$ is negative, the first factor of the product is negative, while the second one is positive.
If $\omega <0$, then the second factor is negative, while the first one is positive, because
$$
(\alpha-1)\gamma-\omega
= (\alpha-1)\gamma - \alpha\beta S_0+\gamma
= \alpha \beta (S_\herd-S_0) < 0,
$$
with $S_\herd$ defined in \eqref{eq63}.
In any case, the zero-order term in \eqref{eq889} is negative when $\omega\neq 0$.

The case $\omega=0$ is similar, with \eqref{eq889} replaced by
\begin{equation}
\psi(T)
\label{eq892}
= -1 + (\alpha-1)\gamma D + 1 + \eta(I_0)
= (\alpha-1)\gamma D + \eta(I_0).
\end{equation}
As the higher-order terms $\eta(I_0)$ in \eqref{eq889} and \eqref{eq892} vanish when $I_0$ goes to 0 {\em uniformly on any compact of $[0,+\infty)$}, this demonstrates \eqref{eq887}.
\end{proof}

\section{Conclusion}
\label{se5}

Voluntarily ignoring many features important in the effective handling of a human epidemic (unmodeled sources of heterogeneity in the spread of the disease, limited hospital capacity, imprecise epidemiological data, partial respect of the enforcement measures\dots), we investigated here the effects of social distancing on a simple SIR model.
In this simplified setting, we have shown that it is possible to exactly answer the following question: given maximal social distancing intensity and duration (but without prescribed starting date), how can one minimize the epidemic final size, that is the total number of individuals infected during the outbreak?
Our contribution is threefold: we have proved the existence of a unique optimal policy, shown some of its key properties, and demonstrated how to determine it numerically by an easily tractable algorithm.
As an outcome, this provides the best possible policy,
in the worst case where no vaccine or therapy exists.
Numerical computations have been provided that exemplify the theoretical results and 
allowed to tabulate the maximal gain attainable in terms of cumulative number of infected during the outbreak, under various experimental conditions.

It is somewhat intuitive that the best policy achievable by imposing a lockdown of possibly time-varying, but limited, intensity on a time interval of which only the duration is restricted, is reached by enforcing the strictest distancing during the whole time interval.
However, up to our knowledge this had not been proved or conjectured so far.
Moreover,
our results show that the onset of the lockdown is uniquely determined as the unique solution of a numerically tractable equation.

The fact that the optimal control does not begin from the earliest possible time is only an apparent paradox.
As a matter of fact, epidemics behave somehow as wildfires ---the propellant being the susceptible individuals.
On the one hand, attempting to contain the spread too early is pointless, as essentially the same amount of propellant will be present after the end of the intervention, leading ultimately to the same epidemic final size.
On the other hand, acting too late is also useless, as in this case most of the stock of propellant will have been already consumed at the time of the intervention.
The best time to proceed lies in between, somewhere around the peak of the epidemic when the herd immunity threshold is crossed ---typically some weeks after the beginning of the epidemics---,  with larger or more intense intervention inducing larger mitigation effect.
The results provided allow to determine precisely what is the best time to initiate social distancing.

As a last remark, notice that in the simplified setting considered here, limited hospital capacity or deaths caused by supplementary mortality are ignored.
Among other extensions, we plan to address in the future the issue of minimization of the epidemic final size under adequate constraints.

\section*{Acknowledgments}

The authors express their warm appreciation to Prof.\ Yannick Privat (IRMA, Universit\'e de Strasbourg, France) and Prof.\ Nicolas Vauchelet (LAGA, Universit\'e Sorbonne Paris Nord, France), for valuable discussions and comments during the elaboration of this article.


\bibliographystyle{abbrv}
\bibliography{biblio-JTB}

\begin{thebibliography}{10}

\bibitem{Alvarez:2020aa}
F.~E. Alvarez, D.~Argente, and F.~Lippi.
\newblock A simple planning problem for covid-19 lockdown.
\newblock Technical report, National Bureau of Economic Research, 2020.

\bibitem{Andreasen:2011aa}
V.~Andreasen.
\newblock The final size of an epidemic and its relation to the basic
  reproduction number.
\newblock {\em Bulletin of Mathematical Biology}, 73(10):2305--2321, 2011.

\bibitem{Behncke:2000aa}
H.~Behncke.
\newblock Optimal control of deterministic epidemics.
\newblock {\em Optimal Control Applications and Methods}, 21(6):269--285, 2000.

\bibitem{Bliman:2020aa}
P.-A. Bliman, M.~Duprez, Y.~Privat, and N.~Vauchelet.
\newblock Optimal immunity control by social distancing for the {SIR} epidemic
  model, 2020.

\bibitem{Buonomo:2019ab}
B.~Buonomo, R.~Della~Marca, and A.~d'Onofrio.
\newblock Optimal public health intervention in a behavioural vaccination
  model: the interplay between seasonality, behaviour and latency period.
\newblock {\em Mathematical Medicine and Biology: A Journal of the IMA},
  36(3):297--324, 2019.

\bibitem{Buonomo:2019aa}
B.~Buonomo, P.~Manfredi, and A.~d'Onofrio.
\newblock Optimal time-profiles of public health intervention to shape
  voluntary vaccination for childhood diseases.
\newblock {\em Journal of Mathematical Biology}, 78(4):1089--1113, 2019.

\bibitem{Djidjou-Demasse:2020aa}
R.~Djidjou-Demasse, Y.~Michalakis, M.~Choisy, M.~T. Sofonea, and S.~Alizon.
\newblock Optimal covid-19 epidemic control until vaccine deployment.
\newblock {\em medRxiv}, 2020.

\bibitem{Katriel:2012aa}
G.~Katriel.
\newblock The size of epidemics in populations with heterogeneous
  susceptibility.
\newblock {\em Journal of Mathematical Biology}, 65(2):237--262, 2012.

\bibitem{Kermack:1927aa}
W.~O. Kermack and A.~G. McKendrick.
\newblock Contributions to the mathematical theory of epidemics---{I}.
\newblock {\em Proceedings of the Royal Society}, 115A:700--721, 1927.

\bibitem{Ketcheson:2020aa}
D.~I. Ketcheson.
\newblock Optimal control of an {SIR} epidemic through finite-time
  non-pharmaceutical intervention, 2020.

\bibitem{Lee:2010aa}
S.~Lee, G.~Chowell, and C.~Castillo-Ch{\'a}vez.
\newblock Optimal control for pandemic influenza: the role of limited antiviral
  treatment and isolation.
\newblock {\em Journal of Theoretical Biology}, 265(2):136--150, 2010.

\bibitem{Lenhart:2007aa}
S.~Lenhart and J.~T. Workman.
\newblock {\em Optimal control applied to biological models}.
\newblock CRC press, 2007.

\bibitem{Lin:2010aa}
F.~Lin, K.~Muthuraman, and M.~Lawley.
\newblock An optimal control theory approach to non-pharmaceutical
  interventions.
\newblock {\em BMC infectious diseases}, 10(1):32, 2010.

\bibitem{Ma:2006aa}
J.~Ma and D.~J. Earn.
\newblock Generality of the final size formula for an epidemic of a newly
  invading infectious disease.
\newblock {\em Bulletin of Mathematical Biology}, 68(3):679--702, 2006.

\bibitem{Manfredi:2013aa}
P.~Manfredi and A.~D'Onofrio.
\newblock {\em Modeling the interplay between human behavior and the spread of
  infectious diseases}.
\newblock Springer Science \& Business Media, 2013.

\bibitem{Miclo:2020aa}
L.~Miclo, D.~Spiro, and J.~Weibull.
\newblock Optimal epidemic suppression under an {ICU} constraint, 2020.

\bibitem{Miller:2012aa}
J.~C. Miller.
\newblock A note on the derivation of epidemic final sizes.
\newblock {\em Bulletin of Mathematical Biology}, 74(9):2125--2141, 2012.

\bibitem{Morris:2020aa}
D.~H. Morris, F.~W. Rossine, J.~B. Plotkin, and S.~A. Levin.
\newblock Optimal, near-optimal, and robust epidemic control, 2020.

\bibitem{Salje:2020aa}
H.~Salje, C.~T. Kiem, N.~Lefrancq, N.~Courtejoie, P.~Bosetti, J.~Paireau,
  A.~Andronico, N.~Hoz{\'e}, J.~Richet, C.-L. Dubost, et~al.
\newblock Estimating the burden of {SARS-CoV-2 in France}.
\newblock {\em Science}, May 2020.

\bibitem{Sharomi:2017aa}
O.~Sharomi and T.~Malik.
\newblock Optimal control in epidemiology.
\newblock {\em Annals of Operations Research}, 251(1-2):55--71, 2017.

\bibitem{Worldometer:aa}
Worldometer.
\newblock Coronavirus {Cases in France}.

\bibitem{Yan:2008aa}
X.~Yan and Y.~Zou.
\newblock Optimal and sub-optimal quarantine and isolation control in {SARS}
  epidemics.
\newblock {\em Mathematical and Computer Modelling}, 47(1-2):235--245, 2008.

\end{thebibliography}

\appendix

\section{Implementation issues}
\label{se6}

Two algorithms were implemented and compared to compute numerically the optimal solution of Problem \eqref{OCP2}, showing as expected identical solutions and validating the theoretical derivations.
Algo.\ \ref{algo:1} uses the fact that \eqref{OCP2} is equivalent to minimizing the function $j$ given in \eqref{eq:opt j}, which is decreasing and then increasing\footnote{This property has been verified graphically by curves not reproduced here.}, see Remark \ref{rmk1}.
The minimum is then calculated by a trisection method.
Algo.\ \ref{algo:2} solves directly equation \eqref{eq:psi=0} of Theorem \ref{th:T} using a bisection method.
In both strategies, all ODE solutions have been computed by a Runge-Kutta fourth-order method.

\begin{algorithm}[H]
\caption{Solving Problem \eqref{OCP2} by trisection method applied to \eqref{eq:opt j}}
\label{algo:1}
\begin{algorithmic}[1]
\REQUIRE{$k\in\mathbb{N}^*$}

\STATE\textbf{Initialization:} $T_{\min,0}=0$ and $T_{\max,0}=S^{-1}(S_{\herd})$ with $(S,I)$ solution to \eqref{SIR} for $u\equiv 1$

\FOR{$i=1,\dots, k$}
\STATE $T_{\text{left}}= T_{\min,i-1}+(T_{\max,i-1}-T_{\min,i-1})/3$
\STATE  $T_{\text{right}}=T_{\min,i-1}+2(T_{\max,i-1}-T_{\min,i-1})/3$
 \STATE Compute $(S_{\text{left}},I_{\text{left}})$, $(S_{\text{right}},I_{\text{right}})$ solutions to \eqref{SIR} 
 for $u_{T_{\text{left}},T_{\text{left}}+D}$, $u_{T_{\text{right}},T_{\text{right}}+D}$
  
  \IF{$j(T_{\text{right}})\geq j(T_{\text{left}})$}
  \STATE $T_{\min,i}= T_{\text{left}}$
  and $T_{\max,i}=T_{\max,i-1}$
  \ELSE\STATE $T_{\min,i}=T_{\min,i-1}$
and  $T_{\max,i}= T_{\text{right}}$
  \ENDIF
\ENDFOR
\RETURN $T_{k}:=(T_{\max,k}+T_{\min,k})/2$, $u_{k}:=u_{T_{k},T_k+D}$
\end{algorithmic}
\end{algorithm}

\begin{algorithm}[H]
\caption{Solving Problem  \eqref{OCP2} by bisection method applied to \eqref{eq:psi=0}}
\label{algo:2}
\begin{algorithmic}[1]
\REQUIRE{$k\in\mathbb{N}^*$}

\STATE\textbf{Initialization:} $T_{\min,0}=0$ and $T_{\max,0}=S^{-1}(S_{\herd})$
 with $(S,I)$ solution to \eqref{SIR} for $u\equiv 1$

\FOR{$i=1,\dots, k$}
  \STATE Let  $T_{\text{test},i}=(T_{\max,i-1}+T_{\min,i-1})/2$
  \IF{$\psi(T_{\text{test},i})\geq 0$}
  \STATE $T_{\max,i}= T_{\text{test},i-1}$
  \ELSE\STATE  $T_{\min,i}= T_{\text{test},i-1}$
  \ENDIF
\ENDFOR
\RETURN $T_{k}:=(T_{\max,k}+T_{\min,k})/2$, $u_{k}:=u_{T_{k},T_k+D}$
\end{algorithmic}
\end{algorithm}

\end{document}